\documentclass[a4paper,12pt]{article}

\usepackage[latin1]{inputenc}
\usepackage{eucal}
\usepackage{exscale}
\usepackage{latexsym}
\usepackage[T1]{fontenc}
\usepackage{amsmath}
\usepackage{graphicx}
\usepackage{amsfonts,makeidx,color,amsmath, amssymb,graphics,mathrsfs,eurosym,amsthm}
\usepackage{epsfig}

\newtheorem{theorem}{Theorem}[section]
\newtheorem{lemma}{Lemma}[section]
\newtheorem{proposition}{Proposition}[section]
\theoremstyle{remark}

\def\RR{\mathbb{R}}

\title{\bf Projection methods based on spline quasi-interpolation for Urysohn integral equations}
\author{Catterina Dagnino, Angelo Dallefrate and Sara Remogna \thanks{Department of Mathematics, University of Torino, via C. Alberto 10, 10123 Torino, Italy ({\tt catterina.dagnino@unito.it, angelodallefrate@libero.it, sara.remogna@unito.it})}}

\date{}
\begin{document}

\maketitle

\begin{abstract}
In this paper we propose projection methods based on spline quasi-interpo\-lating projectors of degree $d$ and class $C^{d-1}$ on a bounded interval for the numerical solution of nonlinear integral equations. We prove that they have high order of convergence $2d+2$ if $d$ is odd and $2d+3$ if $d$ is even. We also present the implementation details of the above methods. Finally, we provide numerical tests, that confirm the theoretical results. Moreover, we compare the theoretical and numerical results with those obtained by using a collocation method based on the same spline quasi-interpolating projectors.
\end{abstract}\vskip 10pt


{\sl Subject classification AMS (MOS)}: 65R20, 65J15, 65D07

\section{Introduction}\label{intro}

Integral equations occur in different fields of mathematical physics, engineering and mechanics. In particular, many problems like physical applications, potential theory and electrostatics are reduced to the solution of nonlinear integral equations, as noticed in \cite{AST} and references therein. One of the most important kinds of nonlinear integral equations is the Urysohn one, defined in the following form
\begin{equation}\label{e}
x-K(x)=f,
\end{equation}
where $K$ is the Urysohn integral operator 
$$
K(x)(s)=\int_{0}^{1}k(s,t,x(t))dt, \qquad s\in [0,1], \qquad x\in X=C[0,1].
$$
The kernel $k(s, t, u)$ and the right hand side $f$ are real valued continuous functions and we assume that (\ref{e}) has a unique solution $\varphi$. We remark that this kind of equation includes the Hammerstein one.

Classical methods to solve nonlinear integral equations are the projection ones. The most popular ones are the Galerkin and the collocation methods, based on a sequence $\{\pi_n\}$ of orthogonal projectors and interpolatory projectors, respectively, onto finite dimensional subspaces $X_n$ approximating $X$ (see e.g. \cite{AP,G} and references therein). Other important methods to numerically solve (\ref{e}) are the Nystr\"om ones (see e.g \cite{ASST,AST,A} and references therein).

Recently, in \cite{GKV}, a modified projection method, providing high order of convergence with respect to classical projection methods, has been proposed to solve (\ref{e}), by using both orthogonal projectors and interpolatory projectors in the space of piecewise polynomials of degree $d$ at most continuous. 

In this paper we intend to use the logical scheme proposed in \cite{GKV}, with the same smoothness requirement for  $k$, $f$ and $\varphi$, in order to construct projection methods based on spline quasi-interpolation projectors (abbr. QIPs) of degree $d$ and class $C^{d-1}$. We show that such methods have high order of convergence $2d+2$ if $d$ is odd and $2d+3$ if $d$ is even. 

We remark that, recently, the use of the spline quasi-interpolation has been proved to work well for the approximation of the solution of linear integral equations (see e.g. \cite{ASS1,ASS2,ASS,DR,DRS}). In particular, in \cite{ASS1} a degenerate kernel method based on (left and right) partial approximation of the kernel by a quartic spline quasi-interpolant is provided. In \cite{ASS2}, the authors propose and analyse a collocation method and a modified Kulkarni's scheme based on spline quasi-interpolating operators, which are not projectors. In \cite{DRS} quadratic and cubic quasi-interpolating projectors are proposed and analysed in Galerkin, Kantorovich, Sloan and Kulkarni schemes. Finally, in \cite{ASS,DR}, quasi-interpolating operators have been presented for the numerical solution of 2D and surface integral equations, respectively.

Here is an outline of the paper. In Section \ref{pro} we introduce the spline QIPs, presenting their convergence properties, and we propose our spline quasi-interpolating projection methods for solving (\ref{e}), studying their convergence order and analysing the implementation details. Finally, in Section \ref{num} we provide some numerical results, illustrating the approximation properties of the proposed methods. Moreover, we compare the theoretical and numerical results with those obtained by using a classical collocation method based on the same spline quasi-interpolating projectors.

\section{Spline quasi-interpolating projection \\ methods}\label{pro}

Before presenting the spline QIPs and the methods for solving (\ref{e}), we introduce some assumptions and definitions. 

Let $\varphi$ be the unique solution of (\ref{e}) and let $a$ and $b$ be two real numbers such that 
$$
[\min_{s\in[0,1]}\varphi(s), \max_{s\in[0,1]}\varphi(s)] \subset (a,b).
$$
Define
$$
\Omega=[0,1]\times[0,1]\times[a,b]. 
$$
Let $\alpha \geq 1$. We assume that $k\in C^{\alpha}(\Omega)$, $\frac{\partial k}{\partial x} \in C^{2\alpha}(\Omega)$,  $f\in C^{\alpha}[0,1]$. Therefore, $K$ is a compact operator from $C[0,1]$ to $C^{\alpha}[0,1]$ and $\varphi \in C^{\alpha}[0,1]$.

The operator $K$ is Fr\'echet differentiable and the Fr\'echet derivative is given by
$$
(K'(x)h)(s)=\int_{0}^{1}\frac{\partial k}{\partial u}(s,t,x(t))h(t)dt. 
$$
For $\delta_{0}>0$, let
$$
B(\varphi,\delta_{0})=\{\psi\in X: \|\varphi-\psi\|_{\infty}<\delta_{0}\}.
$$
Since by assumption $\frac{\partial k}{\partial x}\in C^{2\alpha}(\Omega)$, it follows that $K'$ is Lipschitz continuous in a neighborhood $B(\varphi,\delta_{0})$ of $\varphi$, that means, there exists a constant $\gamma$ such that
$$
\left\|K'(\varphi)-K'(x)\right\| \leq \gamma \left\|\varphi-x\right\|_{\infty}, \qquad    x\in B(\varphi,\delta_0).
$$
The operator $K'(\varphi)$ is compact and we assume that 1 is not an eigenvalue of $K'(\varphi)$.

\subsection{Spline quasi-interpolating projectors}\label{pro2}

Let $X_n=\mathcal{S}^{d-1}_d(I,\mathcal{T}_n)$ be the space of splines of degree $d$ and class $C^{d-1}$ on the uniform knot sequence $\mathcal{T}_n:=\{t_i=ih, 0\le i\le n\}$, with $h=1/n$.

We set $s_{i} := \frac{1}{2}(t_{i-1}+t_{i}),$ for $1\leq i \leq n$ and define the set of quasi-interpolation nodes $\Xi_n=\{\xi_{i}\}_{i=0}^{2n}$ with $\xi_{2i} :=t_{i},$ for $0 \leq i \leq n$ and $\xi_{2i-1} :=s_{i},$ for $1 \leq i \leq n$ and let \textbf{$x_{i}=x(\xi_{i})$}, for $0\leq i \leq 2n$. Let $\pi_n$ be a bounded QIP on $\mathcal{S}^{d-1}_d(I,\mathcal{T}_n)$, i.e. for $x\in C[0,1]$, we can write $\pi_n x$ as
\begin{equation}\label{propn}
	\pi_n x=\sum_{i=1}^{N} \lambda_i(x) B_i,
\end{equation}
where $N=$dim$(\mathcal{S}^{d-1}_d(I,\mathcal{T}_n))=n+d$, the $B_i$'s are the B-splines with support  $[ t_{i-d-1},t_{i}]$, on the usual extended knot sequence $\mathcal{T}_{n}^e=\mathcal{T}_{n} \cup \{t_{-d} = \ldots = t_{0} = 0; 1= t_{n} = \ldots = t_{n+d}\}$, and are a basis for $\mathcal{S}^{d-1}_d(I,\mathcal{T}_n)$, and the coefficients $\lambda_i(x)$ are local functionals having the following form
\begin{equation}\label{def_l}
	\lambda_i(x)=\sum_{j=1}^{F_i} \sigma_{i,j} x_j,
\end{equation}
using discrete values of $x$ in supp$(B_i)$ and the $\sigma_{i,j}$'s are chosen such that $\pi_n x=x$, $\forall x \in \mathcal{S}^{d-1}_d(I,\mathcal{T}_n)$. 

We can refer to \cite{LS} for a general theory of quasi-interpolating operators and examples of QIPs can be found in \cite{dB,DRS}. 

Since the operators $\pi_{n}$ are projectors with a norm which is uniformly bounded, classical results in approximation theory provide 
$$ 
\|x-\pi_{n}x\|_{\infty} \leq C\, {\rm dist}(x, \mathcal{S}^{d-1}_d(I,\mathcal{T}_n)),
$$
where $C=1+\left\|\pi_{n}\right\|_\infty$.

Therefore, using the Jackson type theorem for splines \cite{dB}, we can conclude that there exist constants $\overline{C}_{j}$, depending on $C$ and $j$, such that for all $x\in C^{j}[0,1]$,
$$
\left\|x-\pi_n x\right\|_\infty \leq \overline{C}_j h^j \omega(x^{(j)},h), \quad {\rm with} \quad 0\le j\le d,
$$
where $\omega$ is the modulus of continuity of $x^{(j)}$. In particular for $j=d$ and when $x$ has the derivative of order $d+1$ continuous, we obtain 
\begin{equation}\label{mag}
\|x-\pi_n x\|_{\infty}= O(h^{d+1}).
\end{equation}
The spline projectors of even degree $d$ have the particularly interesting property shown in the following proposition, where we assume QIPs with coefficient functionals $\lambda_i(x)$, $i=d+1, \ldots, n$, such that the values $\sigma_{i,j}$, in (\ref{def_l}), associated with quasi-interpolation nodes symmetric with respect to the center of the support of $B_i$ are equal. Moreover, the coefficient functionals $\lambda_i(x)$, $i=1, \ldots, d$ have symmetry properties with respect to the functionals $\lambda_i(x)$, $i=n+1, \ldots, n+d$, analogous to the symmetry properties of the corresponding B-splines. Both requirements are common choices in the literature for the QIP construction.

\begin{proposition}\label{prop2} 
When $d$ is even, for any function $g\in W^{1,1}$ (i.e. with $\Vert g'\Vert_1$ bounded), if $\left\|x^{(d+2)}\right\|_\infty$ is bounded, then
\begin{equation}\label{mm}
\left|\int_0^1 g(t)(\pi_n x(t)-x(t))dt\right|=O(h^{d+2}).
\end{equation}
\end{proposition}

\begin{proof} We have that 
\begin{equation}\label{m1}
\begin{array}{lll}
\left|\displaystyle\int_0^1 g(t)(\pi_n x(t)-x(t))dt\right| & = &\left|\displaystyle\sum_{j=1}^n \int_{t_{j-1}}^{t_j} g(t)(\pi_n x(t)-x(t))dt\right| \\
& \leq & \displaystyle\sum_{j=1}^n \int_{t_{j-1}}^{t_j} \left|g(t)\right| \left|\pi_n x(t)-x(t) \right|dt\\
& = &\displaystyle\underbrace{\sum_{j=d+1}^{n-d} \int_{t_{j-1}}^{t_j} \left|g(t)\right| \left|\pi_n x(t)-x(t) \right|dt}_{(\circ)} \\
& & + \underbrace{\sum_{j=1}^{d} \int_{t_{j-1}}^{t_j} \left|g(t)\right| \left|\pi_n x(t)-x(t) \right|dt}_{(\diamondsuit)} \\
& & +\underbrace{\sum_{j=n-d+1}^{n} \int_{t_{j-1}}^{t_j} \left|g(t)\right| \left|\pi_n x(t)-x(t) \right|dt}_{(\triangle)}\\
\end{array}
\end{equation}
In order to bound the term $(\circ)$, we can generalize the logical scheme used in \cite[Lemma 4.3]{DRS} for a quadratic QIP, obtaining 
\begin{equation}\label{m2}
(\circ) = O(h^{d+2}).
\end{equation}
Instead, for $(\diamondsuit)$ and $(\triangle)$, taking into account (\ref{mag}), we have
\begin{equation}\label{m3}
(\diamondsuit) + (\triangle) \leq 2dh\left\|g\right\|_\infty \left\|\pi_n x-x\right\|_\infty=O(h^{d+2}).
\end{equation}
Therefore, from (\ref{m1}), (\ref{m2}) and (\ref{m3}), we obtain (\ref{mm}). \end{proof}

\subsection{Description of the methods}\label{meth}

Given the QIP operator $\pi_n: C[0,1] \rightarrow \mathcal{S}^{d-1}_d(I,\mathcal{T}_n)$, as presented in Section \ref{pro2}, the classical collocation method consists in approximating (\ref{e}) by
\begin{equation}\label{col2}
\varphi_{n}^{C}-\pi_{n}K(\varphi_{n}^{C})=\pi_{n}f.
\end{equation}
Instead, in order to obtain a projection method with high order of convergence, we apply the logical scheme proposed in \cite{GKV} and the approximate solution $\varphi_{n}^{H}$ is obtained by solving 
\begin{equation}\label{kul2}
\varphi_{n}^{H}-K_{n}^{H}(\varphi_{n}^{H})=f,
\end{equation}
where 
\begin{equation}\label{expre}
K_{n}^{H}(x)=\pi_{n}K(x)+K(\pi_{n}x)-\pi_{n}K(\pi_{n}x).
\end{equation}
The following theorem states an important result about the local existence and uniqueness of the solution $\varphi_{n}^{H}$ of (\ref{kul2}). The proof is omitted because our QIPs satisfy the hypothesis of the operators considered in \cite{G}, where a general proof is given.

\begin{theorem}\label{thmG} 
Suppose that $\varphi$ is the unique solution of (\ref{e}) with $f=0$ and that 1 is not an eigenvalue of $K'(\varphi).$ Then, there exists a neighborhood $B(\varphi,\delta_{0})$ of $\varphi$ which contains, for all $n$ large enough, a unique solution $\varphi_{n}^{H}$ of (\ref{kul2}). In addition,
$$
\frac{2}{3}\alpha_{n}\leq\|\varphi_{n}^{H}-\varphi\|_{\infty}\leq 2\alpha_{n}, 
$$
where $\alpha_{n}=\|(I-(K_{n}^{H})'(\varphi))^{-1}(K(\varphi)-K_{n}^{H}(\varphi))\|$ is a sequence converging to zero. Also,
$$
\frac{\alpha_{n}}{\|\pi_{n}\varphi-\varphi\|}_{\infty}\to 0,  \qquad n\to \infty.
$$
\end{theorem}

\subsection{Orders of convergence}\label{oc}

Now we study the order of convergence of the proposed methods (\ref{col2}) and (\ref{kul2}). 

We consider a lemma and a theorem which are generalizations of those given in \cite{GKV}. We do not report their proofs, because the QIPs $\pi_n$ here considered satisfy the properties required in \cite{GKV}. Moreover, we propose a theorem that shows a superconvergence phenomenon in the case $d$ even, when the kernel of $K$ is sufficiently smooth.

\begin{lemma}\label{po}
For $\alpha \geq 1$, let $k \in C^{\alpha}(\Omega), \frac{\partial k}{\partial u}\in C^{2\alpha}(\Omega)$ and $f\in C^{\alpha}[0,1]$. Let $\pi_n: C[0,1] \rightarrow \mathcal{S}^{d-1}_d(I,\mathcal{T}_n)$ be a spline QIP operator of kind (\ref{propn}). Then
$$
\left\|(I-\pi_{n})\left[K(\pi_{n}\varphi)-K(\varphi)-K'(\varphi)(\pi_{n}\varphi-\varphi)\right]\right\|_{\infty}=O(h^{3\beta}),
$$
with $\beta=\min\{\alpha,d+1\}$.
\end{lemma}

\begin{theorem}\label{thmord} 
For $\alpha \geq 1$, let $k \in C^{\alpha}(\Omega), \frac{\partial k}{\partial u}\in C^{2\alpha}(\Omega)$ and $f\in C^{\alpha}[0,1]$. Let $\varphi$ be the unique solution of (\ref{e}) and assume that 1 is not an eigenvalue of $K'(\varphi)$. Let $\pi_n: C[0,1] \rightarrow \mathcal{S}^{d-1}_d(I,\mathcal{T}_n)$ be a spline QIP operator of kind (\ref{propn}). Let $\varphi_{n}^{H}$ be the unique solution of (\ref{kul2}). Then
$$
\| \varphi_{n}^{H}-\varphi\|_{\infty}=O(h^{2\beta}),
$$
with $\beta=\min\{\alpha,d+1\}$.
\end{theorem}

If the kernel of $K$ is sufficiently smooth, that is, $\alpha \geq d+1$, we have
$$
\beta=d+1
$$
and hence, if we consider the collocation method (\ref{col2}), as we expect, from (\ref{mag}), we have the following order of convergence:
\begin{equation}\label{estg}
\|\varphi_{n}^{C}-\varphi\|_{\infty}=O(h^{d+1}).
\end{equation}
Concerning the method with high order of convergence (\ref{kul2}), the following result holds.

\begin{theorem}\label{thmordbest} 
For $\alpha \geq d+1$, let $k \in C^{\alpha}(\Omega), \frac{\partial k}{\partial u}\in C^{2\alpha}(\Omega)$ and $f\in C^{\alpha}[0,1]$. Let $\varphi$ be the unique solution of (\ref{e}) and assume that 1 is not an eigenvalue of $K'(\varphi)$. Let $\pi_n: C[0,1] \rightarrow \mathcal{S}^{d-1}_d(I,\mathcal{T}_n)$ be a spline QIP operator of kind (\ref{propn}). Let $\varphi_{n}^{H}$ be the unique solution of (\ref{kul2}). Then
$$
\| \varphi_{n}^{H}-\varphi\|_{\infty}= \left\{\begin{array}{ll}
        O(h^{2d+2}), & \mbox{if $d$ is odd} \\
        O(h^{2d+3}), & \mbox{if $d$ is even and $\varphi$ satisfies the hypothesis of Proposition \ref{prop2}} 
        \end{array}\right. .
$$
\end{theorem}

\begin{proof} If $d$ is odd, the result is an immediate consequence of Theorem \ref{thmord} with $\beta=d+1$. 

Now we suppose $d$ even. We know that, by assumption, $I-K'(\varphi)$ is invertible. From (\ref{expre}), we have
$$
(K_{n}^{H})'(\varphi)=\pi_{n}K'(\varphi)+(I-\pi_{n})K'(\pi_{n}\varphi)\pi_{n}.
$$
Consequently,
$$
K'(\varphi)-(K_{n}^{H})'(\varphi)= (I-\pi_{n})K'(\varphi)(I-\pi_{n})+(I-\pi_{n})(K'(\varphi)-K'(\pi_{n}\varphi))\pi_{n}.
$$
Therefore,
$$
\|K'(\varphi)-(K_{n}^{H})'(\varphi)\| \leq \|(I-\pi_{n})K'(\varphi)(I-\pi_{n})\|+\|(I-\pi_{n})(K'(\varphi)-K'(\pi_{n}\varphi))\pi_{n}\|.
$$
Since $\pi_{n}$ converges to the identity operator pointwise on $C[0,1]$ and $K'(\varphi)$ is compact, it follows that $\|(I-\pi_{n})K'(\varphi)\|\to 0$, as $n\to \infty$.

Since $K'$ is Lipschitz continuous in a neighborhood $B(\varphi,\delta_{0})$ of $\varphi$, we get
$$
\|K'(\varphi)-K'(\pi_{n}\varphi)\| \leq \gamma\|\varphi-\pi_{n}\varphi\|_{\infty}\to 0, \quad \mbox{ as } n\to\infty.
$$
Thus, since the sequence $(\|\pi_{n}\|)$ is uniformly bounded, 
$$
\|K'(\varphi)-(K_{n}^{H})'(\varphi)\| \to 0 \quad \mbox{ as } n\to\infty.
$$
It follows that $I-(K_{n}^{H})'(\varphi)$ is invertible, for $n$ big enough, and  
$$
\|(I-(K_{n}^{H})'(\varphi))]^{-1}\| \leq 2\|(I-K'(\varphi))^{-1}\|.
$$
By using Theorem \ref{thmG}, we obtain
\begin{equation}\label{quatt}
\|\varphi_{n}^{H}-\varphi\|_{\infty}\leq 4\|(I-K'(\varphi))^{-1}\|\|(I-\pi_{n})(K(\varphi)-K(\pi_{n}\varphi))\|_{\infty}.
\end{equation}
Considering
$$
\begin{array}{ll}
(I-\pi_{n})(K(\varphi)-K(\pi_{n}(\varphi)))= & \underbrace{-(I-\pi_{n})[K(\pi_{n}\varphi)-K(\varphi)-K'(\varphi)(\pi_{n}\varphi-\varphi)]}_{(\Box)}\\
&   \underbrace{-(I-\pi_{n})K'(\varphi)(\pi_{n}\varphi-\varphi)}_{(\Diamond)},
\end{array}
$$
by Lemma \ref{po}, with $\beta=d+1$, we have
\begin{equation}\label{estim}
\|(\Box)\|_{\infty}= O(h^{3d+3}).
\end{equation}
Moreover, from the approximation properties of $\pi_n$ stated in Section \ref{pro2}, we get
\begin{equation}\label{dd}
\|(\Diamond)\|_{\infty} \leq C\|(K'(\varphi)(\pi_{n}\varphi-\varphi))^{(d+1)}\|_{\infty}h^{d+1}.
\end{equation}
Since $\frac{\partial k}{\partial u}\in C^{2\alpha}(\Omega)$, it follows that
$$
(K'(\varphi)(\pi_{n}\varphi-\varphi))^{(d+1)}(s)=\int_{0}^{1} \frac{\partial k^{d+2}}{\partial s^{d+1} \partial u}(s,t,\varphi(t))(\pi_{n}\varphi-\varphi)(t)dt,
$$
then, by using Proposition \ref{prop2} with $g(t)=\frac{\partial k^{d+2}}{\partial s^{d+1} \partial u}(s,t,\varphi(t))$, we get
\begin{equation}\label{dd1}
\|(K'(\varphi)(\pi_{n}\varphi-\varphi))^{(d+1)}\|_{\infty} = O(h^{d+2}).
\end{equation}
Thus, from (\ref{dd}) and (\ref{dd1})
\begin{equation}\label{estim1}
\|(\Diamond)\|_{\infty} = O(h^{2d+3}).
\end{equation}
So, from (\ref{quatt}), (\ref{estim}) and (\ref{estim1}), we conclude that
$$
\|\varphi_{n}^{H}-\varphi\|_{\infty}=O(h^{2d+3}).
$$\end{proof}

\subsection{Implementation details}\label{id}
\subsubsection{Spline projection method with high order of convergence}
From (\ref{kul2}) and (\ref{expre}), we have
$$
\varphi_{n}^{H}-\pi_{n}K(\varphi_{n}^{H})-K(\pi_{n}\varphi_{n}^{H})+\pi_{n}K(\pi_{n}\varphi_{n}^{H})=f.
$$
After some algebra, taking into account that $\pi_n$ is a projector, we obtain 
$$
\pi_{n}\varphi_{n}^{H}-\pi_{n}K(\varphi_{n}^{H})=\pi_{n}f,
$$
\begin{equation}\label{phi}
\varphi_{n}^{H}=\pi_{n}\varphi_{n}^{H}+(I-\pi_{n})(K(\pi_{n}\varphi_{n}^{H})+f)
\end{equation}
and
\begin{equation}\label{m}
\pi_{n}\varphi_{n}^{H}-\pi_{n}K(\pi_{n}\varphi_{n}^{H}+(I-\pi_{n})(K(\pi_{n}\varphi_{n}^{H})+f))=\pi_{n}f.
\end{equation}
Define
$$
\psi_{n}=\pi_{n}\varphi_{n}^{H}
$$
and
\begin{equation}\label{deff}
F_{n}(y)=y-\pi_{n}K(y+(I-\pi_{n})(K(y)+f))-\pi_{n}f,   \quad y\in \mathcal{S}^{d-1}_d(I,\mathcal{T}_n),
\end{equation}
whose Fr\'echet derivative is given by
\begin{equation}\label{derf}
(F_{n})'(y)h=h-\pi_{n}K'(y+(I-\pi_{n})(K(y)+f))(I+(I-\pi_{n})K'(y))h.
\end{equation}
The equation (\ref{m}) becomes 
$$
\psi_{n}-\pi_{n}K(\psi_{n}+(I-\pi_{n})(K((\psi_{n})+f))=\pi_{n}f,
$$
which is equivalent to $F_{n}(\psi_{n})=0$ and it is iteratively solved by applying the Newton-Kantorovich method.

Given an initial approximation $\psi_{n}^{(0)}$, the iterates $\psi_{n}^{(k)}, k=0,1,2,\dots,$ are given by
$$
\psi_{n}^{(k+1)}=\psi_{n}^{(k)}-((F_{n})'(\psi_{n}^{(k)}))^{-1}F_{n}(\psi_{n}^{(k)}),
$$
i.e.
\begin{equation}\label{n}
((F_{n})'(\psi_{n}^{(k)}))\psi_{n}^{(k+1)}=((F_{n})'(\psi_{n}^{(k)}))\psi_{n}^{(k)}-F_{n}(\psi_{n}^{(k)}).
\end{equation}
\newline
Define, according to (\ref{phi}),
\begin{equation}\label{phink}
\varphi_{n}^{(k)}=\psi_{n}^{(k)}+(I-\pi_{n})(K(\psi_{n}^{(k)})+f).
\end{equation}
Then, from (\ref{deff}) and (\ref{derf}), (\ref{n}) can be written as
\begin{equation}\label{so}
\begin{array}{c}
\psi_{n}^{(k+1)}-\pi_{n}K'(\varphi_{n}^{(k)})\psi_{n}^{(k+1)}-\pi_{n}K'(\varphi_{n}^{(k)})(I-\pi_{n})K'(\psi_{n}^{(k)})\psi_{n}^{(k+1)}\\
=\pi_{n}(K(\varphi_{n}^{(k)})+f)-\pi_{n}K'(\varphi_{n}^{(k)})\psi_{n}^{(k)}-\pi_{n}K'(\varphi_{n}^{(k)})(I-\pi_{n})K'(\psi_{n}^{(k)})\psi_{n}^{(k)}.
\end{array}
\end{equation}
Since $\psi_{n}^{(k)}\in \mathcal{S}^{d-1}_d(I,\mathcal{T}_n)$, we can write
\begin{equation}\label{phivet}
\psi_{n}^{(k)}=\sum_{j=1}^{N}x_{n}^{(k)}(j)B_{j}, \quad x_{n}^{(k)} \in \RR^N.
\end{equation}
Taking into account (\ref{phivet}), (\ref{so}) can be written as
\begin{equation}\label{lu}
\begin{array}{ll}
&\displaystyle\sum_{i=1}^{N}x_{n}^{(k+1)}(i)B_{i}-\pi_{n}K'(\varphi_{n}^{(k)})\sum_{j=1}^{N}x_{n}^{(k+1)}(j)B_{j}\\
&-\pi_{n}K'(\varphi_{n}^{(k)})(I-\pi_{n})K'(\psi_{n}^{(k)})\displaystyle\sum_{j=1}^{N}x_{n}^{(k+1)}(j)B_{j}\\
=&\pi_{n}K(\varphi_{n}^{(k)})+\pi_{n}f-\pi_{n}K'(\varphi_{n}^{(k)})\displaystyle\sum_{j=1}^{N}x_{n}^{(k)}(j)B_{j}\\
&-\pi_{n}K'(\varphi_{n}^{(k)})(I-\pi_{n})K'(\psi_{n}^{(k)})\displaystyle\sum_{j=1}^{N}x_{n}^{(k)}(j)B_{j}.
\end{array}
\end{equation}
Now, we apply the definition of the operator $\pi_{n}$ given in (\ref{propn}), and by identifying the coefficients of $B_{i}$, from (\ref{lu}), there results
\begin{equation}\label{lu2}
\begin{array}{ll}
&x_{n}^{(k+1)}(i)-\lambda_{i}\left(K'(\varphi_{n}^{(k)})\displaystyle\sum_{j=1}^{N}x_{n}^{(k+1)}(j)B_{j}\right)\\
&-\lambda_{i}\left(K'(\varphi_{n}^{(k)})(I-\pi_{n})K'(\psi_{n}^{(k)})\displaystyle\sum_{j=1}^{N}x_{n}^{(k+1)}(j)B_{j}\right)\\
=&\lambda_{i}\left(K(\varphi_{n}^{(k)})\right)+\lambda_{i}(f)-\lambda_{i}\left(K'(\varphi_{n}^{(k)})\displaystyle\sum_{j=1}^{N}x_{n}^{(k)}(j)B_{j}\right)\\
&-\lambda_{i}\left(K'(\varphi_{n}^{.(k)})(I-\pi_{n})K'(\psi_{n}^{(k)})\displaystyle\sum_{j=1}^{N}x_{n}^{(k)}(j)B_{j}\right), \quad    i=1,\ldots,N.
\end{array}
\end{equation}
Finally, by using the linearity of $\pi_n$ and of the coefficient functionals $\lambda_i$, we can write (\ref{lu2}) as follows:
\begin{equation}\label{lu3}
\begin{array}{ll}
&x_{n}^{(k+1)}(i)-\displaystyle\sum_{j=1}^{N}x_{n}^{(k+1)}(j)\lambda_{i}(K'(\varphi_{n}^{(k)})B_{j})\\
&-\displaystyle\sum_{j=1}^{N}x_{n}^{(k+1)}(j)\lambda_{i}\left(K'(\varphi_{n}^{(k)})(I-\pi_{n})K'(\psi_{n}^{(k)})B_{j}\right)\\
=&\lambda_{i}\left(K(\varphi_{n}^{(k)})\right)+\lambda_{i}(f)-\displaystyle\sum_{j=1}^{N}x_{n}^{(k)}(j)\lambda_{i}\left(K'(\varphi_{n}^{(k)})B_{j}\right)\\
&-\displaystyle\sum_{j=1}^{N}x_{n}^{(k)}(j)\lambda_{i}\left(K'(\varphi_{n}^{(k)})(I-\pi_{n})K'(\psi_{n}^{(k)})B_{j}\right), \quad   i=1,\ldots,N.
\end{array}
\end{equation}
Therefore, we have obtained the system of linear equations (\ref{lu3}) of size $N$, whose matrix form is
\begin{equation}\label{p}
\left(I-A_{n}^{(k)}-B_{n}^{(k)}\right)x_{n}^{(k+1)}=d_{n}^{(k)},
\end{equation}
where, for $i,j=1,2,\ldots,N$,
\begin{itemize}
	\item $A_{n}^{(k)}(i,j)=\lambda_{i}\left(K'(\varphi_{n}^{(k)})B_{j}\right)$,
	\item $B_{n}^{(k)}(i,j)= \lambda_{i}\left(K'(\varphi_{n}^{(k)})(I-\pi_{n})K'(\psi_{n}^{(k)})B_{j}\right)$,
	\item $d_{n}^{(k)}(i)=\lambda_{i}\left(K(\varphi_{n}^{(k)})\right)+\lambda_{i}(f)-(A_{n}^{(k)}x_{n}^{(k)})(i)-(B_{n}^{(k)}x_{n}^{(k)})(i)$,
\end{itemize}
and $\varphi_{n}^{(k)}$ is given by (\ref{phink}).

\subsubsection{Spline collocation method}
The equation (\ref{e}) is approximated by
$$
\varphi_{n}^{C}-\pi_{n}K(\varphi_{n}^{C})=\pi_{n}f
$$
and hence $\varphi_{n}^{C}\in \mathcal{S}^{d-1}_d(I,\mathcal{T}_n)$. Define
$$
G_{n}(y)=y-\pi_{n}K(y)-\pi_{n}f, \hspace{3mm} y\in \mathcal{S}^{d-1}_d(I,\mathcal{T}_n)
$$
and solve
$$
G_{n}(\varphi_{n}^{C})=0
$$
iteratively by using the Newton-Kantorovich method. Let $\zeta_{n}^{(0)}$ be an initial approximation and the iterates $\zeta_{n}^{(k)}$, $k=0,1,2,\ldots,$ are
\begin{equation}\label{r}
\zeta_{n}^{(k+1)}-\pi_{n}K'(\zeta_{n}^{(k)})\zeta_{n}^{(k+1)}=\pi_{n}(K(\zeta_{n}^{(k)})+f)-\pi_{n}K'(\zeta_{n}^{(k)})\zeta_{n}^{(k)}.
\end{equation}
Let
$$
\zeta_{n}^{(k)}=\sum_{j=1}^{N}y_{n}^{(k)}(j)B_{j}, \quad y_{n}^{(k)} \in \RR^n,
$$
then (\ref{r}) is equivalent to the following system of linear equations of size $N$
\begin{equation}\label{q}
\left(I-C_{n}^{(k)}\right)y_{n}^{(k+1)}=r_{n}^{(k)},
\end{equation}
where, for $i,j=1,2,\ldots,N$,
\begin{itemize}
\item $C_{n}^{(k)}(i,j)=\lambda_{i}\left(K'(\zeta_{n}^{(k)})B_{j}\right)$,
\item $r_{n}^{(k)}(i)=\lambda_{i}\left(K(\zeta_{n}^{(k)})\right)+\lambda_{i}(f)-(C_{n}^{(k)}y_{n}^{(k)})(i)$. 
\end{itemize}

We remark that a comparison of (\ref{p}) and (\ref{q}) shows that the latter system is much simpler. Indeed, in the former it is necessary to construct an additional matrix and the right hand side has an extra term. 
Moreover, we notice that the elements of the matrix $C_{n}^{(k)}$ are similar to those of $A_{n}^{(k)}$, with $\zeta_{n}^{(k)}$ instead of $\varphi_{n}^{(k)}$.

\section{Numerical results}\label{num}

In this section we present two test equations, that are also considered in \cite{AST,GKV}. In \cite{AST}, the authors propose a superconvergent Nystr\"om method and, in \cite{GKV}, a modified projection method providing high order of convergence. In the numerical tests, both papers make use of projection methods based on discontinuous piecewise constant and $C^0$ piecewise linear polynomials.

Here, we consider both the space $\mathcal{S}^{1}_2(I,\mathcal{T}_n)$ of $C^1$ quadratic splines and $\mathcal{S}^{2}_3(I,\mathcal{T}_n)$ of $C^2$ cubic splines. In the first case we use two QIPs, proposed in \cite{DRS} and in \cite[p. 155]{dB}, denoted by $Q_{2},Q_{2}^{dB}$, respectively. For the cubic case we use the QIP constructed in \cite{DRS}, denoted by $Q_{3}$. For details concerning their definition and construction see \cite{dB,DRS}.

We recall that $Q_{2}$ is superconvergent on the set of evaluation points $\Xi_n$ \cite{DRS}. It is easy to verify the same property also for $Q_{2}^{dB}$. Therefore the following proposition holds.

\begin{proposition} If $\|x^{(4)}\|_{\infty}$ is bounded, then, for $\pi_{n}=Q_{2},Q_{2}^{dB}$
$$
\left|\pi_{n}x(\xi_i)-x(\xi_i)\right|=O(h^{4}), \qquad 0\leq i \leq 2n.
$$
\end{proposition}

Therefore, from the above proposition we get
\begin{equation}\label{sup_punt} 
\begin{array}{l}
|\varphi(\xi_i)-\varphi_{n}^{C}(\xi_i)|=O(h^{4}),\\
|\varphi(\xi_i)-\varphi_{n}^{H}(\xi_i)|=O(h^{8}),
\end{array}
\end{equation}
for the methods (\ref{col2}) and (\ref{kul2}), respectively, with $\pi_{n}=Q_{2},Q_{2}^{dB}$.

The integrals appearing in (\ref{p}) and (\ref{q}), in the matrices $A_{n}^{(k)},B_{n}^{(k)},C_{n}^{(k)}$, in the vectors $d_{n}^{(k)},r_{n}^{(k)}$ and in $\varphi_{n}^{(k)}$ are computed numerically with high accuracy, by using a classical composite $m$-point Gauss-Legendre quadrature formula with $m=20$.

For all the tests, for increasing values of $n$, we compute the following maximum absolute error
$$
E_{\infty}^{\mu}=\max_{v\in \mathcal{G}}|\varphi(v)-\varphi_{n}^{H}(v)|,
$$
where $\mathcal{G}$ is a set of 1500 equally spaced points in $[0,1]$ and $\mu=H_{2}$, $H_{2}^{dB}$, $H_{3}$ in case of the method (\ref{kul2}) based on the spline operators $Q_{2}$, $Q_{2}^{dB}$, $Q_{3}$. We compute also
$$
E_{\infty}^{\mu}=\max_{v\in \mathcal{G}}|\varphi(v)-\varphi_{n}^{C}(v)|,
$$
where $\mu=C_{2}$, $C_{2}^{dB}$, $C_{3}$, in case of the method (\ref{col2}) based on the same above operators. For each error we compute the corresponding numerical convergence order $O_{\infty}^{\mu}$, obtained by the logarithm to base 2 of the ratio between two consecutive errors.

Moreover, we compute the maximum absolute error at the quasi-interpolation nodes
$$
ES^{\mu}=\max_{0\leq i\leq 2n}|\varphi(\xi_{i})-\varphi_{n}^{H}(\xi_{i})|,
$$
with $\mu=H_{2}$, $H_{2}^{dB}$ in case of the method (\ref{kul2}) based on the spline operators $Q_{2}$, $Q_{2}^{dB}$, for increasing values of $n$. Similarly, we define
$$
ES^{\mu}=\max_{0\leq i\leq 2n}|\varphi(\xi_{i})-\varphi_{n}^{C}(\xi_{i})|, 
$$
with $\mu=C_{2}$, $C_{2}^{dB}$ in case of collocation method (\ref{col2}) based on the same above operators. For each error we compute the corresponding numerical convergence order $O^{\mu}$. 

These numerical tests confirm the theoretical results proved in Section \ref{oc}. We remark that the presented methods provide an approximate solution of class $C^{1}$ when $X_n=\mathcal{S}_{2}^{1}(I,\mathcal{T}_{n})$ and of class $C^{2}$ when $X_n=\mathcal{S}_{3}^{2}(I,\mathcal{T}_{n})$.

\subsection*{Test 1}
Consider the following Hammerstein integral operator with a degenerate kernel, defined as follows
$$
K(x)(s)=\int_{0}^{1}p(s)q(t)x^{2}(t)dt, \hspace{5mm}s\in[0,1],
$$
where
$$
p(s)=cos(11\pi s), \hspace{4mm} q(t)=sin(11\pi t).
$$
Then $K$ is compact and $\varphi-K(\varphi)=f$ has a unique solution for $f\in C[0,1]$. We choose
$$
f(s)=\Bigg(1-\frac{2}{33\pi}\Bigg)cos(11\pi s), \hspace{2mm} s\in[0,1],
$$
so that
$$
\varphi(s)=cos(11\pi s), \hspace{2mm} s\in[0,1].
$$

By using computational procedures constructed in the Matlab environment, we obtain the results reported in Tables \ref{tab:1}, \ref{tab:2} and \ref{tab:3}, that confirm the theoretical ones stated in Theorem \ref{thmordbest} for the projection method with high order of convergence and in (\ref{estg}) for the spline collocation method. In particular, we have 
$$
\begin{array}{l}
E_{\infty}^{H_2}, E_{\infty}^{H_2^{dB}}= O(h^7), \quad E_{\infty}^{H_3}= O(h^8),\\
E_{\infty}^{C_2}, E_{\infty}^{C_2^{dB}}= O(h^3), \quad E_{\infty}^{C_3}= O(h^4),
\end{array}
$$
and, thanks to (\ref{sup_punt})
$$
ES_{\infty}^{H_2}, ES_{\infty}^{H_2^{dB}}= O(h^8), \quad ES_{\infty}^{C_2}, ES_{\infty}^{C_2^{dB}}= O(h^4).
$$ 

\begin{table}[ht!]
\caption{Spline projection method with high order of convergence and spline collocation method based on $Q_{2}$}\label{tab:1}
\begin{center}
\begin{tabular}{|c||cc|cc|cc|cc|cc|cc|}
\hline
$n $    & $E^{H_{2}}_{\infty}$ & $O^{H_{2}}_{\infty}$ & $ES^{H_{2}}$ & $O^{H_{2}}$& $E^{C_{2}}_{\infty}$ & $O^{C_{2}}_{\infty}$ & $ES^{C_{2}}$ & $O^{C_{2}}$\\
\hline
40  & 1.08(-06) &     & 6.97(-07) &     & 7.74(-03) &     & 4.98(-03) &            \\ 
\hline
80  & 4.08(-09) & 8.1 & 2.26(-09) & 8.2 & 6.77(-04) & 3.5 & 3.76(-04) & 3.7			\\
\hline
160 & 2.13(-11) & 7.6 & 6.31(-12) & 8.5 & 8.17(-05) & 3.0 & 2.43(-05) & 4.0	\\
\hline
320 & 1.42(-13) & 7.2	& 2.14(-14) & 8.2	& 1.01(-05) & 3.0	& 1.53(-06) & 4.0		\\
\hline
640 &           &     &-          & -		& 1.26(-06) & 3.0	& 9.57(-08) & 4.0         \\
\hline
\end{tabular}
\end{center}
\end{table}

\begin{table}[ht!]
\caption{Spline projection method with high order of convergence and spline collocation method based on $Q_{2}^{dB}$}\label{tab:2}
\begin{center}
\begin{tabular}{|c||cc|cc|cc|cc|}
\hline
$n $    & $E^{H_{2}^{dB}}_{\infty}$ & $O^{H_{2}^{dB}}_{\infty}$&$ES^{H_{2}^{dB}}$ & $O^{H_{2}^{dB}}$ & $E^{C_{2}^{dB}}_{\infty}$ & $O^{C_{2}^{dB}}_{\infty}$ & $ES^{C_{2}^{dB}}$ & $O^{C_{2}^{dB}}$\\
\hline    
40  & 1.50(-06) &     & 1.41(-06) &     & 9.12(-03) &     & 8.58(-03) &           \\
\hline
80  & 1.12(-08) & 7.1 & 7.79(-09) & 7.5 & 7.97(-04) & 3.5 & 5.54(-04) & 4.0		\\
\hline
160 & 8.07(-11) & 7.1 & 3.28(-11) & 7.9 & 8.58(-05) & 3.2 & 3.49(-05) & 4.0		\\
\hline
320 & 6.13(-13) & 7.0	& 1.32(-13) & 8.0 & 1.02(-05) & 3.0	& 2.19(-06) & 4.0      \\
\hline
640 & 5.33(-15) & 6.8	& -         & -	  & 1.27(-06) & 3.0 & 1.37(-07) & 4.0                   \\
\hline
\end{tabular}
\end{center}
\end{table}

\begin{table}[ht!]
\caption{Spline projection method with high order of convergence and spline collocation method based on $Q_{3}$}\label{tab:3}
\begin{center}
\begin{tabular}{|c||cc|cc|}
\hline
$n $    & $E^{H_{3}}_{\infty}$ & $O^{H_{3}}_{\infty}$ & $E^{C_{3}}_{\infty}$ & $O^{C_{3}}_{\infty}$ \\
\hline
40  & 2.38(-08) &     & 1.53(-03) &      \\
\hline
80  & 9.40(-11) & 8   & 9.27(-05) & 4.0  \\
\hline
160 & 1.12(-13) & 9.7 & 5.58(-06) & 4.1    \\
\hline
320 & -         & -	  & 3.43(-07) & 4.0		 \\
\hline
640 & -         & -   & 1.34(-08) & 4.7	   \\
\hline
\end{tabular}
\end{center}
\end{table}

\subsection*{Test 2}
Consider the following Urysohn integral equation
$$
\varphi(s)-\int_{0}^{1}\frac{dt}{s+t+\varphi(t)}=f(s), \hspace{5mm} 0\leq s\leq 1,
$$
where $f$ is chosen so that $\varphi(t)=\frac{1}{t+c}$, $c>0$, is a solution.

We consider $c=1$, $c=0.1$ and we remark that the exact solution is ill behaved in the case $c=0.1$.

Firstly, we consider $c=1$ and we obtain the results presented in Tables \ref{tab:4}, \ref{tab:5} and \ref{tab:6}.

\begin{table}[ht!]
\caption{Spline projection method with high order of convergence and spline collocation method based on $Q_{2}$, $c=1$}\label{tab:4}
\begin{center}
\begin{tabular}{|c||cc|cc|cc|cc|}
\hline
$n $    & $E^{H_{2}}_{\infty}$ & $O^{H_{2}}_{\infty}$ & $ES^{H_{2}}$ & $O^{H_{2}}$& $E^{C_{2}}_{\infty}$ & $O^{C_{2}}_{\infty}$ & $ES^{C_{2}}$ & $O^{C_{2}}$  \\
\hline
4    &  8.48(-08) &    & 5.06(-08) &     & 6.85(-04) &     & 3.84(-04) & \\
\hline
8    & 7.84(-10) & 6.8 & 3.50(-10) & 7.2 & 9.54(-05) & 2.8 & 3.84(-05) & 3.3\\
\hline
16   & 5.08(-12) & 7.3 & 1.47(-12) & 7.9 & 1.21(-05) & 3.0 & 3.10(-06) & 3.6\\
\hline
32   & 3.08(-14) & 7.4 & 5.55(-15) & 8.0 & 1.50(-06) & 3.0 & 2.21(-07) & 3.8\\
\hline
64   & -         &  -  &   -       &   - & 1.85(-07) & 3.0 & 1.48(-08) & 3.0  \\
\hline
\end{tabular}
\end{center}
\end{table}

\begin{table}[ht!]
\caption{Spline projection method with high order of convergence and spline collocation method based on $Q_{2}^{dB}$, $c=1$}\label{tab:5}
\begin{center}
\begin{tabular}{|c||cc|cc|cc|cc|}
\hline
$n $    & $E^{H_2^{dB}}_{\infty}$ & $O^{H_2^{dB}}_{\infty}$ & $ES^{H_2^{dB}}$ & $O^{H_2^{dB}}$& $E^{C_2^{dB}}_{\infty}$ & $O^{C_2^{dB}}_{\infty}$ & $ES^{C_2^{dB}}$ & $O^{C_2^{dB}}$\\
\hline    
4    & 2.00(-07) &     & 1.33(-07) &     & 7.27(-04) &     & 4.59(-04) &  \\
\hline
8    & 2.75(-09) & 6.2 & 1.42(-09) & 6.6 & 9.96(-05) & 2.9 & 4.67(-05) & 3.3\\
\hline
16   & 2.67(-11) & 6.7 & 9.34(-12) & 7.2 & 1.25(-05) & 3.0 & 3.84(-06) & 3.6\\
\hline
32   & 2.24(-13) & 6.9 & 4.60(-14) & 7.7 & 1.53(-06) & 3.0 & 2.78(-07) & 3.8\\
\hline
64   & -         &  -  & -         &   - & 1.87(-07) & 3.0 & 1.87(-08) & 3.9      \\
\hline
\end{tabular}
\end{center}
\end{table}

\begin{table}[ht!]
\caption{Spline projection method with high order of convergence and spline collocation method based on $Q_{3}$, $c=1$}\label{tab:6}
\begin{center}
\begin{tabular}{|c||cc|cc|}
\hline
$n $    & $E^{H_{3}}_{\infty}$ & $O^{H_{3}}_{\infty}$  & $E^{C_{3}}_{\infty}$ & $O^{CC_{3}}_{\infty}$\\
\hline
4      & 1.58(-09) &     & 9.02(-05) &      \\
\hline
8      & 3.30(-12) & 8.9 & 7.77(-06) & 3.5   \\
\hline
16     & 6.55(-15) & 9.0 & 6.84(-07) & 3.5  \\
\hline
32     & -        & -    & 5.00(-08) & 3.8		  \\
\hline
64     & -        & -    & 3.36(-09) & 3.9		\\
\hline
\end{tabular}
\end{center}
\end{table}

When we consider $c=0.1$ by using the same procedures we get the results in the Tables \ref{tab:7}, \ref{tab:8} and \ref{tab:9}.

\begin{table}[ht!]
\caption{Spline projection method with high order of convergence and spline collocation method based on $Q_{2}$, $c=0.1$}\label{tab:7}
\begin{center}
\begin{tabular}{|c||cc|cc|cc|cc|}
\hline
$n $    & $E^{H_{2}}_{\infty}$ & $O^{H_{2}}_{\infty}$ & $ES^{H_{2}}$ & $O^{H_{2}}$& $E^{C_{2}}_{\infty}$ & $O^{C_{2}}_{\infty}$& $ES^{C_{2}}$ & $O^{C_{2}}$  \\
\hline
4    &  4.50(-07)  & &  1.13(-07) &   & 6.80(-01)  &  &  5.51(-01)  & \\
\hline
8    &  3.87(-10) &  10.2 & 2.51(-10)  & 8.8& 2.67(-01)  & 1.3  &   2.10(-01)  & 1.4\\
\hline
16    & 1.00(-11)&   5.3 & 1.01(-12)  & 8.0 & 7.04(-02) &  1.9  &5.12(-02) &  2.0\\
\hline
32     & 1.21(-13)  & 6.4 & 1.07(-14) &  6.7& 1.29(-02)  & 2.4  &  7.99(-03)  & 2.7\\
\hline
64      & -    &  -   &       &      & 1.86(-03)   &   2.8    &   8.65(-04)       &   3.2  \\
\hline
\end{tabular}
\end{center}
\end{table}

\begin{table}[ht!]
\caption{Spline projection method with high order of convergence and spline collocation method based on $Q_{2}^{dB}$, $c=0.1$}\label{tab:8}
\begin{center}
\begin{tabular}{|c||cc|cc|cc|cc|}
\hline
$n $    & $E^{H_2^{dB}}_{\infty}$ & $O^{H_2^{dB}}_{\infty}$ & $ES^{H_2^{dB}}$ & $O^{H_2^{dB}}$& $E^{C_2^{dB}}_{\infty}$ & $O^{C_2^{dB}}_{\infty}$& $ES^{C_2^{dB}}$ & $O^{C_2^{dB}}$\\
\hline    
4    & 2.89(-06) &   & 1.39(-06) & &  7.44(-01)   & & 6.41(-01)&   \\
\hline
8    & 1.26(-08)  & 7.8 & 3.43(-09) &  8.7  & 2.94(-01)  & 1.3  & 2.48(-01)  & 1.4\\
\hline
16     &  8.00(-11) &  7.3 &1.22(-11) &  8.1 & 7.77(-02)  & 1.9 &  6.14(-02)   &2.0\\
\hline
32     &  5.88(-13)  & 7.1& 4.80(-14) &  8.0 & 1.42(-02)  & 2.4 &  9.87(-03)&   2.6\\
\hline
64      & -        &  -       & -  &     -      & 2.02(-03)    &  2.8     &   1.11(-03)       &  3.2    \\
\hline
\end{tabular}
\end{center}
\end{table}

\begin{table}[ht!]
\caption{Spline projection method with high order of convergence and spline collocation method based on $Q_{3}$, $c=0.1$}\label{tab:9}
\begin{center}
\begin{tabular}{|c||cc|cc|}
\hline
$n $    & $E^{H_{3}}_{\infty}$ & $O^{H_{3}}_{\infty}$ & $E^{C_{3}}_{\infty}$ & $O^{C_{3}}_{\infty}$ \\
\hline
4     &  1.97(-08)&     &   4.17(-01)     &             \\
\hline
8      & 1.16(-11) &  10.7 & 1.03(-01) &  2.0 \\
\hline
16     & 1.29(-13)  & 6.5 &  1.70(-02)   &2.6\\
\hline
32    & - &  -   & 1.96(-03)  & 3.1 \\
\hline
64   &  -  & -   &  1.75(-04) &  3.5\\
\hline
\end{tabular}
\end{center}
\end{table}

Also in test 2, the theoretical results stated in Theorem \ref{thmordbest} for the projection method with high order of convergence and in (\ref{estg}) for the spline collocation method are confirmed. In particular, we have 
$$
\begin{array}{l}
E_{\infty}^{H_2}, E_{\infty}^{H_2^{dB}}= O(h^7), \quad E_{\infty}^{H_3}= O(h^8),\\
E_{\infty}^{C_2}, E_{\infty}^{C_2^{dB}}= O(h^3), \quad E_{\infty}^{C_3}= O(h^4),
\end{array}
$$
and, thanks to (\ref{sup_punt})
$$
ES_{\infty}^{H_2}, ES_{\infty}^{H_2^{dB}}= O(h^8), \quad ES_{\infty}^{C_2}, ES_{\infty}^{C_2^{dB}}= O(h^4).
$$

\section{Conclusions}
In this paper we have proposed spline projection methods for the numerical solution of nonlinear integral equations. In particular, we have considered spline quasi-interpolating projectors on a bounded interval for defining a projection method with high order of convergence and a collocation method of classical type. We have studied their order of convergence and we have analysed the implementation details. Finally, we have presented some numerical examples, illustrating the approximation properties of the proposed methods. The next research step, that is a work in progress, is the study of nonlinear integral equations with non smooth kernels.

\section*{Acknowledgements} 
The authors thank the University of Torino for its support to their research.



\end{document}